\documentclass{amsart}
\usepackage{amsmath,dsfont}
\usepackage{amsthm}
\usepackage{amssymb}
\usepackage{mathrsfs}
\usepackage{graphicx}
\usepackage{enumerate}
\usepackage{tikz}
\usepackage[colorlinks,
linkcolor=red,
anchorcolor=blue,
citecolor=green
]{hyperref}
\newtheorem{theorem}{Theorem}[section]

\newtheorem{prop}{Proposition}
\newtheorem{lemma}[theorem]{Lemma}
\newtheorem{remark}[theorem]{Remark}

\numberwithin{equation}{section}
\title[Spectral restriction bounds for orthonormal systems]{Sharp Spectral-Cluster Restriction Bounds for Orthonormal Systems
}
\author{Changbiao Jian}
\address{School of Mathematics and Statistics, Guangxi Normal University, Guilin, Guangxi 541004, PR China}
\email{bobjian1@gmail.com}
\author{Xing Wang}
\address{School of Mathematics, Hunan University, Changsha, HN 410012, PR China}
\email{xingwang@hnu.edu.cn}
\author{Yakun Xi}
\address{School of Mathematical Sciences, Zhejiang University, Hangzhou 310027, PR China}
\email{yakunxi@zju.edu.cn}

\begin{document}
\begin{abstract}
Let $\Sigma$ be a smooth submanifold of a compact Riemannian manifold $M$.
We study the restriction to $\Sigma$ of densities generated by
$L^2(M)$-orthonormal systems of eigenfunctions whose eigenfrequencies lie in
a unit-width spectral cluster. For every $p\geq2$, we establish
$L^{p/2}(\Sigma)$ bounds that quantify the gain arising from orthogonality
and determine the optimal summability exponent for the coefficient sequence
in all dimensions and codimensions. The summability exponent is fully sharp, while the frequency dependence is
optimal modulo logarithmic factors. In the
codimension-one case when $\dim M\geq3$, the optimal summability exponent
exhibits a new, dimension-independent critical point at $p=4$, in addition
to the classical critical point $p=\frac{2d}{d-1}$. This phenomenon has no
analogue in the corresponding restriction estimates for a single
eigenfunction. We also obtain improved, essentially sharp estimates when
$\Sigma$ is a curve with nonvanishing geodesic curvature on a Riemannian
surface.
\end{abstract}
\maketitle

\section{Introduction}

Laplace--Beltrami eigenfunctions may be viewed as stationary states of a single-particle quantum system. From a mathematical perspective, one studies the distribution and concentration of their mass \cite{Berry1977,Shnirelman1974,ColinDeVerdiere1985,Zelditch1987,sogge1988,bgt2007,Zelditch2009QuantumChaos}; from a physical perspective, the corresponding questions are commonly formulated in terms of the localization and delocalization of quantum states \cite{Anderson1958,Berry1977,Heller1984Scars,EversMirlin2008}. In recent years, this viewpoint has been extended to interacting many-body systems, where many-body localization and many-body quantum scars provide two distinct mechanisms for nonthermal behavior \cite{BaskoAleinerAltshuler2006,AbaninAltmanBlochSerbyn2019,BernienEtAl2017,TurnerEtAl2018}.

In free fermionic systems, the Pauli exclusion principle requires the occupied single-particle states to be mutually orthogonal, and the associated one-particle density is  
$$  
\rho(x)=\sum_j |f_j(x)|^2.  
$$ 
Densities formed by summing over orthogonal states are also the fundamental objects appearing in the Lieb--Thirring inequalities. These inequalities control an integral power of the density by the total kinetic energy of the system, thereby quantitatively expressing how the Pauli exclusion principle suppresses excessive spatial concentration of fermions \cite{LiebThirring1975,LiebThirring1976,Frank2023Orthonormal}. In a narrow high-energy frequency window, spectral clusters therefore provide a natural setting in which to study such orthogonal states. This leads to a question fundamentally different from estimates for a single eigenfunction: to what extent can multiple mutually orthogonal high-frequency states concentrate simultaneously along the same submanifold, and how strongly does orthogonality suppress such collective concentration?

We first recall the corresponding restriction estimates for a single eigenfunction. Let $(M,g)$ be a compact boundaryless Riemannian manifold of dimension $d\geq2$, and let $\Delta_g$ be the corresponding Laplace--Beltrami operator on $M$.  Let $e_\lambda$ be an $L^2$-normalized eigenfunction of $\Delta_g$, that is,
 \[-\Delta_g e_\lambda=\lambda^2e_\lambda,\quad \text{and }\int_M |e_\lambda|^2\, dV_g=1.\]

Burq--G\'erard--Tzvetkov \cite{bgt2007} established $L^p$ estimates for the restriction of eigenfunctions to a submanifold. In what follows, let $\Sigma\subset M$ be a smoothly embedded submanifold of dimension $k$. Define
${\delta(k,d,p)}$ by
\begin{align*}
	&\delta(k,d,p)=\begin{cases}
		\dfrac{d-1}{4}-\dfrac{d-2}{2p},       &\text{ if }k=d-1,\; 2\leq p \le \dfrac{2d}{d-1} ,\\
		\dfrac{d-1}{2}-\dfrac{d-1}{p},    &\text{ if }k=d-1,\; \dfrac{2d}{d-1}< p\leq\infty,\\
		\dfrac{d-1}{2}-\dfrac{k}{p},     &\text{ if }1\leq k\leq d-2\text{ and } 2\leq p\leq\infty.
	\end{cases}
\end{align*}
Burq--G\'erard--Tzvetkov proved the sharp restriction bound
\begin{equation}\label{lp restriction}\|e_\lambda\|_{L^p(\Sigma)}\le C\lambda^{{\delta(k,d,p)}}\|e_\lambda\|_{L^2(M)},\end{equation}
up to an additional loss of $(\log\lambda)^{\frac{1}{2}}$ when
$(p,k)=(\frac{2d}{d-1},d-1)$ or $(p,k)=(2,d-2)$.
 
 Hu \cite{hu2009} subsequently gave another proof of \eqref{lp restriction} and removed the $\log$ loss for the case $(p,k)=(\frac{2d}{d-1},d-1)$. More recently, the second author and Zhang \cite{wang2021} removed the $\log$ loss for totally geodesic submanifolds and curves with nonvanishing curvature in the $(p,k)=(2,d-2)$ case. The bounds \eqref{lp restriction} have been improved and generalized under various geometric assumptions and operator settings; see, e.g., \cite{chen2014,xi2017,zhang2017,blair2018,xi2019,HWZ2026Schrodinger,gao2024refined}. Moreover, certain $L^p(\Sigma)$ restriction bounds have been connected to the $L^p(M)$ norms of eigenfunctions via Kakeya--Nikodym norms; see, e.g., \cite{bourgain2009,sogge2011,bs2015,bs2015b,miao2016,bs2017}.

{For $\lambda\ge2$, let
\[
I_\lambda=[\lambda,\lambda+1),\qquad
\Pi_\lambda=\mathds{1}_{I_\lambda}(\sqrt{-\Delta_g}),
\]
and let \[E_\lambda=\operatorname{Ran}\Pi_\lambda\]
denote the associated unit-width spectral cluster.
Let $\{f_j\}_{j\in J}\subset E_\lambda$ be an orthonormal system of
eigenfunctions, with eigenfrequencies in $I_\lambda$, and let
$\{t_j\}_{j\in J}\subset\mathbb C$ be finitely supported.

To state our result uniformly, define
\[
\alpha(k,d,p):=
\begin{cases}\max\left\{2,\dfrac{p}{2}\right\},
& d=2,\ k=1,\\[6pt]
\dfrac{p}{2},
& d\geq3,\ 1\leq k\leq d-2,\\[6pt]
\dfrac{2p}{p+2},
& d\geq3,\ k=d-1,\ 
2\leq p<\dfrac{2d}{d-1},\\[8pt]
\dfrac{2p(d-2)}{4d-p-4},
& d\geq3,\ k=d-1,\ 
\dfrac{2d}{d-1}\leq p\leq4,\\[8pt]
\dfrac{p}{2},
& d\geq3,\ k=d-1,\ 4<p\leq\infty,
\end{cases}
\]
and
\[
h(k,d,p):=
\begin{cases}
\min\left\{\dfrac12,\dfrac2p\right\},
& d=2,\ k=1,\\[8pt]
0,
& d\geq3,\ 1\leq k\leq d-3,\\[4pt]
\dfrac{2}{p},
& d\geq3,\ k=d-2,\\[8pt]

0,
& d\geq3,\ k=d-1,\ 
2\leq p<\dfrac{2d}{d-1},\\[8pt]
\dfrac{2d-p}{p(d-2)},
& d\geq3,\ k=d-1,\ 
\dfrac{2d}{d-1}\leq p\leq4,\\[8pt]
\dfrac2p,
& d\geq3,\ k=d-1,\ 4<p\leq\infty.
\end{cases}
\]

\begin{theorem}\label{main theorem}
Suppose that $d\geq2$ and $1\leq k\leq d-1$. For every
$p\in[2,\infty]$ and $1\leq\alpha\leq\alpha(k,d,p)$, there exists a
constant $C_\alpha>0$ such that
\begin{equation}\label{main estimate}
\Biggl\|\sum_{j\in J}t_j|f_j|^2\Biggr\|_{L^{p/2}(\Sigma)}
\leq
C_\alpha\,(\log\lambda)^{h(k,d,p)}
\lambda^{2\delta(k,d,p)}\,
\|\{t_j\}\|_{l^\alpha(J)}.
\end{equation}
The exponent $\alpha(k,d,p)$ is sharp for every admissible choice of
$k$, $d$, and $p$, while the polynomial power
$\lambda^{2\delta(k,d,p)}$ is optimal modulo logarithmic factors.

Moreover, the logarithmic factor can be removed for
every $1\leq\alpha<\alpha(k,d,p)$ if $d\geq3$, $k=d-2$ and $p>2$, if
$d=2$, $k=1$, and $p\neq4$, or if $d\geq3$, $k=d-1$, and $p>4$.
\end{theorem}

As illustrated by Figures~\ref{fig 1} and~\ref{fig 2}, the hypersurface
case displays a substantially richer structure than the higher-codimension
case, particularly when $d\geq3$. The transition at
$p=\frac{2d}{d-1}$ reflects the familiar kink in the
single-eigenfunction restriction exponent. More strikingly, the
orthonormal problem exhibits a second, dimension-independent transition at
$p=4$. Our sharpness constructions show that this new critical point arises
from the competition between two distinct many-state concentration
mechanisms: an orthonormal family whose combined density fills the
hypersurface and a family of states localized in $\lambda^{-1}$-scale
neighborhoods of separated points on the hypersurface. The former gives the
stronger obstruction below $p=4$, while the latter dominates above $p=4$,
and the two constraints coincide exactly at $p=4$. This additional kink has
no analogue in the classical restriction estimate for a single
eigenfunction and is one of the main new phenomena revealed by our results.}

\begin{figure}
	\begin{tikzpicture}[samples=200,scale=1.0]
	\coordinate (O) at (0,0);
	\coordinate (D) at (2.5,0);
	\coordinate (A) at (0,2.5);
	\coordinate (B) at (0,2.5);
	\coordinate (H) at (0,5);
	\coordinate (J) at (5,0);
	\coordinate (K) at (2.5,2.5);
        \coordinate (Z) at (2.5,5);
	\coordinate (M) at (5,2.5);
		\coordinate (W) at (5,5);
      \coordinate (P) at (0,5);
	\coordinate (XMAX) at (6,0);
	\coordinate (XMIN) at (-1,0);
	\coordinate (YMAX) at (0,6);
	\coordinate (YMIN) at (0,-1);

	\draw[->,>=stealth] (XMIN)--(XMAX) node [below] {$\frac{1}{p}$};
	\draw[->,>=stealth] (YMIN)--(YMAX) node [below left] {};

	\draw [red] (K)--(H)node [left] {$\frac{1}{2}$};
	\draw [red] (K)--(M);
	\draw [blue] (O)--(Z);
 \draw [blue] (W)--(Z);

\draw [dashed] (Z)--(P) ;
	\draw [dashed] (K)--(A) node [left,red] {{$\frac{1}{4}$}};
	\draw [dashed] (M)--(J) node [below] {{$\frac{1}{2}$}};
	\draw [dashed] (Z)--(D) node [below] {{$\frac{1}{4}$}};
		\draw [dashed] (K)--(B) node [left,red] {{$\frac{1}{4}$}};
	
   \draw [dashed] (K)--(J);
			\draw [dashed] (W)--(J) ;

				\draw (W) node [right,blue] {{$1/\alpha(1,2,p)$}};
				\draw (M) node [right,red] {{$\delta(1,2,p)$ }};
	\coordinate(O) at (0,0);
	\end{tikzpicture}
\caption{$\delta(1,2,p)$ and $1/\alpha(1,2,p)$.}\label{fig 1}
\end{figure}

\begin{figure}
\begin{tikzpicture}[samples=200,scale=1.0]
    \coordinate (O) at (0,0);
    \coordinate (XMIN) at (-1,0);
    \coordinate (XMAX) at (6,0);
    \coordinate (YMIN) at (0,-1);
    \coordinate (YMAX) at (0,6);

    \coordinate (Pinf) at (0,0);          
    \coordinate (Pfour) at (2.5,0);       
    \coordinate (Pthree) at (3.333,0);   
    \coordinate (Ptwo) at (5,0);         

    \coordinate (Yquarter) at (0,1.25);   
    \coordinate (Ythird) at (0,1.666);    
    \coordinate (Yhalf) at (0,2.5);       
    \coordinate (Yfive6) at (0,4.166);   
    \coordinate (Yone) at (0,5);         

    \coordinate (R0) at (0,5);            
    \coordinate (R1) at (3.333,1.666);    
    \coordinate (R2) at (5,1.25);         

    \coordinate (B0) at (0,0);           
    \coordinate (B1) at (2.5,2.5);        
    \coordinate (B2) at (3.333,4.166);    
    \coordinate (B3) at (5,5);           
    \draw[->,>=stealth] (XMIN)--(XMAX) node [below] {$\frac{1}{p}$};
    \draw[->,>=stealth] (YMIN)--(YMAX);

    \draw[red] (R0)--(R1)--(R2);

    \draw[blue] (B0)--(B1)--(B2)--(B3);

    \draw[dashed] (Pfour)--(B1);
    \draw[dashed] (Pthree)--(B2);
    \draw[dashed] (Ptwo)--(B3);

    \draw[dashed] (Yquarter)--(R2);
    \draw[dashed] (Ythird)--(R1);
    \draw[dashed] (Yhalf)--(B1);
    \draw[dashed] (Yfive6)--(B2);
    \draw[dashed] (Yone)--(R0);

    \draw (Pfour) node [below] {$\frac{1}{4}$};
    \draw (Pthree) node [below] {$\frac{1}{3}$};
    \draw (Ptwo) node [below] {$\frac{1}{2}$};

    \draw (Yquarter) node [left,red] {$\frac{1}{4}$};
    \draw (Ythird) node [left,red] {$\frac{1}{3}$};
    \draw (Yhalf) node [left,blue] {$\frac{1}{2}$};
    \draw (Yfive6) node [left,blue] {$\frac{5}{6}$};
    \draw (Yone) node [left] {$1$};

    \draw (R0) node [right,red] {$\delta(2,3,p)$};
    \draw (B3) node [right,blue] {$1/\alpha(2,3,p)$};

\end{tikzpicture}
\caption{$\delta(2,3,p)$ and $1/\alpha(2,3,p)$, exhibiting the additional
critical point at $p=4$.}\label{fig 2}
\end{figure}

At the optimal endpoint, the summability exponent is $\alpha(k,d,p)$,
which equals $p/2$ in codimension at least two. Since
$\alpha(k,d,p)>1$ for $p>2$, Theorem~\ref{main theorem} yields a genuine
gain over the trivial estimate obtained by applying the triangle inequality
and the single-eigenfunction restriction bound to each $f_j$ separately.

In the regimes where $\alpha(k,d,p)=p/2$, our estimate also has a natural
vector-valued interpretation. Consider any orthogonal, but not necessarily
normalized, family of eigenfunctions $\{f_j\}_{j\in J}\subset E_\lambda$. Our results imply
\footnote{Possibly modulo a $\log\lambda$--loss.}
\begin{equation}\label{Minkowski}
\Big\|\|\{f_j\}\|_{l^2(J)}\Big\|_{L^p(\Sigma)}
\le C\,\lambda^{\delta(k,d,p)}\,
\Big\|\{\|f_j\|_{L^2(M)}\}\Big\|_{l^p(J)}.
\end{equation}
Thus, in these regimes, Theorem \ref{main theorem} allows us
to reverse the order of the $L^p(\Sigma)$ and $l^p(J)$ norms. By contrast,
a direct application of Minkowski's inequality together with
\eqref{lp restriction} gives
\begin{equation}\label{Minkowski2}
\Big\|\|\{f_j\}\|_{l^2(J)}\Big\|_{L^p(\Sigma)}
\le C\,\lambda^{\delta(k,d,p)}\,
\Big\|\{\|f_j\|_{L^2(M)}\}\Big\|_{l^2(J)}.
\end{equation}
Clearly \eqref{Minkowski} improves \eqref{Minkowski2} for all $p>2$.

Our sharpness constructions apply in all dimensions and codimensions. In the
surface case $d=2$ and $k=1$, Section~\ref{sec 6} gives a more refined
description: on the standard two-sphere, the estimate is saturated for
orthonormal systems of every cardinality
$\#J\sim\lambda^\beta$, $0\leq\beta\leq1$, up to logarithmic factors.

In the case $d=2$, when $\Sigma$ is a curve parametrized by arc-length, Burq--Gérard--Tzvetkov \cite{bgt2007} obtained improved bounds for $2\le p\le4$ provided the geodesic curvature of $\Sigma$ is nonvanishing, i.e.
\[
    \langle \nabla_{\Sigma'(s)}\Sigma'(s), \nabla_{\Sigma'(s)}\Sigma'(s)\rangle \ne 0 \quad \text{for all } s.
\]
They prove that, in this case,
\begin{equation}\label{eq:curved}
    \|e_\lambda\|_{L^p(\Sigma)}
    \le C (1+\lambda)^{\frac{1}{3}-\frac{1}{3p}}
    \|e_\lambda\|_{L^2}.
\end{equation}
As a direct consequence of \eqref{eq:curved} with $p=2$,
\eqref{main estimate} with $p=4$, Minkowski's inequality, and interpolation,
we obtain the following essentially sharp result for orthonormal systems.

\begin{theorem}\label{nonvanishing geodesic curvature}
Suppose that $d=2$ and $k=1$. Let $\Sigma$ be a curve with nonvanishing geodesic curvature. Then for any $p\in[2,4]$ there exists $C>0$ such that
\begin{equation}\label{gf22}
    \Biggl\|\sum_{j\in J} t_j\,|f_j|^{2}\Biggr\|_{L^{p/2}(\Sigma)}
    \le
    C\,
    (\log\lambda)^{1-\frac2p}
    \lambda^{\frac23-\frac2{3p}}\,
    \|\{t_j\}\|_{l^{p/2}(J)}.
\end{equation}
Moreover, \eqref{gf22} is essentially sharp when $M$ is the standard sphere $\mathbb{S}^2$.
\end{theorem}

Our approach combines the operator-theoretic framework of Frank and Sabin with
the oscillatory-integral methods of Burq--G\'erard--Tzvetkov. Following Frank
and Sabin \cite{frank2017,sabin2017}, we use a duality principle to reduce
estimates for orthonormal systems to Schatten norm bounds for suitable weighted
operators. We then apply Sogge's reproducing-operator reduction
\cite{sogge2017}, expressing the localized spectral projector through the
half-wave propagator and its local parametrix. This places the problem in the
oscillatory-integral framework developed by Burq--G\'erard--Tzvetkov
\cite{bgt2007}.

For submanifolds of codimension at least two, a $TT^*$ argument and a dyadic
decomposition of the resulting kernel yield the required Schatten bounds.
The logarithmic losses arise from borderline kernel or dyadic summations,
while the logarithm-free non-endpoint estimates follow from interpolation
with appropriate Schatten class bounds.

This paper is organized as follows. Section \ref{sec 2} recalls the definition
of Schatten classes and the duality principle from \cite{sabin2017}.
Section \ref{sec 3} contains the standard reduction to oscillatory integral
operators. Sections \ref{sec 4} and \ref{sec 5} prove
Theorem \ref{main theorem}. Sections \ref{sec 6} and \ref{sec 7} establish
the sharpness assertions in Theorem \ref{main theorem}. Finally,
Section \ref{sec 8} proves the sharpness statement in
Theorem \ref{nonvanishing geodesic curvature}.

\subsection*{Notation.} Throughout this paper, the symbol $C$ denotes a positive constant whose value may vary from line to line. We write $A\lesssim B$ if
$A\le CB$ for a constant independent of $\lambda$ and of the functions under
consideration; dependence on fixed data is suppressed. We write $A\sim B$ if
$A\lesssim B$ and $B\lesssim A$. For brevity, whenever $k$ and $d$ are clear from the context, we abbreviate
$\delta(k,d,p)$, $\alpha(k,d,p)$, and $h(k,d,p)$ by
$\delta(p)$, $\alpha(p)$, and $h(p)$, respectively.

\subsection*{Acknowledgements}
This project is supported by the National Key Research and Development
Program of China No. 2022YFA1007200. X. W. is partially supported by the Fundamental Research Funds for the Central Universities Grant No. 531118010864. Y. X. is partially supported by Zhejiang Provincial Natural Science Foundation of China under
Grant No. LR25A010001, and NSF China Grant No. 12571107 and 12171424. The authors would like to thank Rupert Frank and Julien Sabin for their kind communications and for pointing out a gap in an earlier version of the paper.

\section{Schatten classes}\label{sec 2}		
In this section, we recall some fundamental properties of Schatten class operators and review a duality principle established by Frank and Sabin.
Let $\mathfrak{H}$ and $\mathfrak{K}$ be complex, separable Hilbert spaces, and denote by $\mathfrak{B}_0(\mathfrak{H, \mathfrak{K}})$  the space of compact linear operators from $\mathfrak{H}$ to $\mathfrak{K}$. For $T\in \mathfrak{B}_0(\mathfrak{H, \mathfrak{K}})$, the operator 
$|T|:=(T^*T)^{1/2}$ is compact and positive. Its eigenvalues, which are non-negative, are called the singular values of $T$, and are arranged in decreasing order as $\sigma_1\geq\sigma_2\geq\cdots\geq0$. For $1\leq p\leq\infty $, the Schatten class $\mathfrak{S}^p(\mathfrak{H},\mathfrak{K})$ consists of all compact operators $T$ whose singular value(counted according to multiplicity) form a sequence in $l^p$. Naturally, the Schatten norm of $T$ is defined by the $l^p$ norm of the singular value sequence, that is,
\[\big\|T\big\|_{\mathfrak{S}^p(\mathfrak{H},\mathfrak{K})}:=\Bigl(\sum_{j\geq1}\sigma^p_j\Bigr)^{1/p},\]
with the standard modification when $p=\infty.$
  When $\mathfrak{H}=\mathfrak{K}$, we write $\mathfrak{S}^p(\mathfrak{H}):=\mathfrak{S}^p(\mathfrak{H},\mathfrak{H})$ and $\mathfrak{B}_0(\mathfrak{H}):=\mathfrak{B}_0(\mathfrak{H},\mathfrak{H})$. For further details, we refer the reader to Simon’s monograph \cite{simon2005}.
  
 We will require the following lemma from \cite[Theorem 2.7]{simon2005}.
  \begin{lemma}[Theorem 2.7 \cite{simon2005}]\label{z1}
    Let $A$ and $C$ be bounded operators on $\mathfrak{H}.$ Then for all $B\in\mathfrak{B}_0(\mathfrak{H})$ and $1\leq p\leq\infty$, we have
$$\|ABC\|_{\mathfrak{S}^p(\mathfrak{H})}\leq\|A\|\|C\|\|B\|_{\mathfrak{S}^p(\mathfrak{H})}.$$
Here $\|{}\cdot{}\|$ denotes the operator norm on $\mathfrak H.$
  \end{lemma}

We now review a well-known duality principle for Schatten class operators, as presented in \cite[Lemma 3]{sabin2017}.
\begin{lemma}[Duality principle]\label{l1}
	Let $\mathfrak{H}$ be a separable Hilbert space. For $2\leq p\leq \infty$, $\alpha\geq1$, with $1/\alpha+1/\alpha'=1/p+1/p'=1$, suppose that $T$ is a bounded operator from $\mathfrak{H}$ to $L^{p}(\mathbb{R}^d)$. Then the following are equivalent.
	\begin{enumerate}[(1)]
		\item  There exists a constant $C>0$ such that
		\begin{equation}\label{a1}
\Big\|WTT^*\overline{W}\Big\|_{\mathfrak{S}^{\alpha'}(L^2(\mathbb{R}^d))}\leq C\|W\|^2_{L^{2p/(p-2)}(\mathbb{R}^d)},     \text{for all }W\in L^{2p/(p-2)}(\mathbb{R}^d,\mathbb{C}).
		\end{equation}
        
		\item  For any orthonormal system $\{f_j\}_{j\in J}$ in
		$\mathfrak{H}$ and any sequence $\{t_j\}_{j\in J}\subset \mathbb{C}$, there is a constant $C'$ such that
		\begin{equation}\label{a2}
		\Biggl\|\sum_{j\in J}t_j|Tf_j|^2\Biggr\|_{L^{p/2}(\mathbb{R}^d)}\leq C'\Biggl(\sum_{j\in J}|t_j|^{\alpha}\Biggr)^{1/\alpha},
		\end{equation}
	 Moreover, the values of the optimal constants $C$ and $C'$ coincide. 
	\end{enumerate} 
\end{lemma}
If one examines the proof of the above lemma in \cite{sabin2017}, one can see that this duality principle is a direct consequence of the duality between Lebesgue $L^p$ spaces.

\section{Approximate projection operators
}\label{sec 3}
In this section, we perform the standard reduction introduced by Sogge \cite{sogge2017} and collect a few key oscillatory‐integral estimates from \cite{bgt2007}.  Sogge’s idea is to consider an operator that reproduces eigenfunctions.  Fix a small constant $\epsilon_0>0$ and let $\lambda\ge2$.  Choose a Schwartz function $\chi$ such that 
\[
\chi(0)=1,\quad \chi(t)>\tfrac12\text{ for }t\in[0,1],
\quad\text{and}\quad
\widehat\chi(t)=0\quad\text{unless }|t|\in[\epsilon_0,2\epsilon_0].
\]
Define
\[
\chi_{\lambda}f
:=\chi\bigl(\sqrt{-\Delta_g}-\lambda\bigr)f
=\frac{1}{2\pi}\int_{|t|\in[\epsilon_0,2\epsilon_0]}
\bigl(e^{it\sqrt{-\Delta_g}}f\bigr)\,e^{-it\lambda}\,\widehat\chi(t)\,dt.
\]
For $\epsilon_0\ll1$ and $|t|\le2\epsilon_0$, a local‐coordinate parametrix shows that $e^{it\sqrt{-\Delta_g}}$ is a Fourier integral operator.  A stationary‐phase argument then gives the following slight variant of \cite[Lemma 5.1.3]{sogge2017}, as in \cite{bgt2007,sogge2011}.

\begin{lemma}\label{l0}
  Let $\epsilon_0>0$ be smaller than one-tenth of the injectivity radius of $(M,g)$.  
  In local coordinates,
  \[
    \chi_{\lambda}f(x)
    =\lambda^{\frac{d-1}{2}}
     \int_{M}e^{\,i\lambda\psi(x,y)}\,a_{\lambda}(x,y)\,f(y)\,dy
    +R_{\lambda}f(x),
  \]
  where
  \[
    {\rm supp\,} a_{\lambda}
    \subset
    \bigl\{(x,y)\colon\tfrac{\epsilon_0}{2}\le d_g(x,y)\le \epsilon_0\bigr\},
    \qquad
    \psi(x,y)=-d_g(x,y),
  \]
  and $a_{\lambda}\in C^{\infty}_0$ satisfies 
  \[
    |\partial_{x,y}^\alpha a_{\lambda}(x,y)|\le C_\alpha
    \quad\text{for all multi‐indices }\alpha.
  \]
  Moreover, for every $N\in\mathbb Z_+$ and $2\le p\le\infty$ there is $C_{N,p}$ so that
  \[
    \|R_{\lambda}\|_{L^2\to L^p}\le C_{N,p}\,\lambda^{-N}.
  \]
\end{lemma}

\begin{remark}
  By a partition of unity we may also assume $a_\lambda(x,y)$ is supported in small coordinate‐chart neighborhoods of fixed points $x_0,y_0\in M$ with 
  $d_g(x,y)\in[\epsilon_0/2,\epsilon_0]$, both lying inside the geodesic ball $B(x_0,10\epsilon_0)$.  We can always take $\epsilon_0>0$ smaller if needed.
\end{remark}
Define
   \[
    T_{\lambda}f(x)
    =
     \int_{M}e^{\,i\lambda\psi(x,y)}\,a_{\lambda}(x,y)\,f(y)\,dy.
  \]
The approximate projection operator $\chi_\lambda$ is usually called a reproducing operator since it reproduces eigenfunctions by $\chi_\lambda e_\lambda=e_\lambda$.  So it suffices to consider
the operator norm estimates of $T_{\lambda}$ to get the eigenfunction estimates, as $R_{\lambda}$ satisfies much
better bounds than we want to prove.

\section{Proof in dimension two}\label{sec 4}
We are now ready to prove Theorem \ref{main theorem} when $k=d-1$ and $d=2$.
\begin{proof}[Proof of Theorem \ref{main theorem} when $k=d-1$ when $d=2$]
We begin by establishing the endpoint estimate \eqref{main estimate}.
Recall that the operator $\chi_{\lambda}$ is defined by
$$\chi_{\lambda}f:=\chi(\sqrt{-\Delta_g}-\lambda)f=\frac{1}{2\pi}\int_{|t|\in[\epsilon_0,2\epsilon_0]}(e^{it\sqrt{-\Delta_g}}f)e^{-it\lambda}\hat{\chi}(t)\,dt.$$
By replacing $t_j$ with $|t_j|$, we may assume that $t_j\geq0$. Since
$\chi_\lambda f_j=\chi(\lambda_j-\lambda)f_j$ and
$\chi(\lambda_j-\lambda)>1/2$, it follows that
\[
\sum_{j\in J}t_j|f_j|^2
\leq
4\sum_{j\in J}t_j|\chi_\lambda f_j|^2.
\]
 Therefore, establishing \eqref{main estimate} in dimension two reduces to proving the following estimate involving $\chi_{\lambda} f_j$:
	\begin{equation}\label{g20}
	\Biggl\|\sum_{j\in J}t_j|\chi_{\lambda} f_j|^2\Biggr\|_{L^{p/2}(\Sigma)}\le C\lambda^{2\delta{(p)}}(\log\lambda)^{h(p)}\Big\|\{t_j\}\Big\|_{l^{\alpha(p)}(J)}.
	\end{equation}
	By Lemma \ref{l1}, this inequality is equivalent to the Schatten norm estimate
	\begin{equation}\label{g00}
\Big\|W{\chi}_{\lambda}\chi^{*}_{\lambda}\overline{W}\Big\|_{\mathfrak{S}^{(\alpha(p))'}(L^2(\Sigma))}\le C\lambda^{2\delta(p)}(\log\lambda)^{h(p)}\|W\|^2_{L^{2p/(p-2)}(\Sigma)},   {\text{for all} }W\in L^{2p/(p-2)}(\Sigma),
	\end{equation}
	where $C$ is independent of $W$ and $\lambda$.

	Note that, by Lemma \ref{l0}, $R_{\lambda}$ is
an operator that always contributes a term rapidly decaying in $\lambda$, it suffices to prove the bound
	\begin{equation}\label{g24}
	\Big\|WT_{\lambda}T^{*}_{\lambda}\overline{W}\Big\|_{\mathfrak{S}^{(\alpha(p))'}(L^2(\Sigma))}\le C\lambda^{2\delta(p)-1}(\log\lambda)^{h(p)}\|W\|^2_{L^{2p/(p-2)}(\Sigma)}.
	\end{equation}
	
	First, consider the case $p=\infty$, which corresponds to $(\alpha(p))'=1$.  Then we have
	\begin{equation}\label{g23}
	\begin{aligned}
\Big\|WT_{\lambda}T^{*}_{\lambda}\overline{W}\Big\|_{\mathfrak{S}^{1}(L^2(\Sigma))}
	&= \|WT_{\lambda} \|^2_{\mathfrak{S}^{2}(L^2(M)\to L^2(\Sigma))}\\
 &=\int_{\Sigma}\int_M|W(x)|^2|T_{\lambda}(x,y)|^2\,dy\,dx\\
 &\leq C\|W\|^2_{L^2(\Sigma)},
	\end{aligned}
	\end{equation}
    where $T_{\lambda}(x,y)=e^{-i\lambda d_g(x,y)}a_{\lambda}(x,y)$ is the kernel function associated with $T_{\lambda}$.
    Here we have used the fact that  if 
  \[
    S : L^2(M)\longrightarrow L^2(\Sigma)
  \]
  is a Hilbert–Schmidt operator with integral kernel $S(x,y)$, then
  \[
    \big\|S\big\|_{\mathfrak S^2(L^2(M)\to L^2(\Sigma))}^2
    =
    \int_{\Sigma}\!\!\int_M \bigl|S(x,y)\bigr|^2\,dy\,dx.
  \]

Next we prove \eqref{g24} for $p=4$.  Once this case is established, the
range $4\le p\le\infty$ follows immediately by interpolation.
Because $\Sigma$ is compact, Hölder’s inequality implies that for every
$2\le p<4$,
\[
  \Biggl\|\sum_{j\in J} t_j\,|\chi_{\lambda}f_j|^{2}\Biggr\|_{L^{p/2}(\Sigma)}
  \;\lesssim\;
  \Biggl\|\sum_{j\in J} t_j\,|\chi_{\lambda}f_j|^{2}\Biggr\|_{L^{2}(\Sigma)}.
\]
Hence \eqref{main estimate} is a consequence of \eqref{g24} with $p=4$.

Assume that we are in the geodesic normal coordinate system about $x_0\in M$, and $\Sigma:[0,1]\to M$
parameterized by the arc length $s$ and passes through $x_0$. Partition of unity allows us to assume
that $\Sigma$ is contained in the geodesic ball centered at $x_0$ with radius $\epsilon_0/10$, i.e., $|x(s)|\leqslant \epsilon_0/10$ and $x(0) = 0$, which forces $\epsilon_0/3\leq |y| \leq 4\epsilon_0/3$.
We write
\[\mathcal{T}_{\lambda}(f)(s) :={T}_{\lambda}(f)(x(s))=\int e^{i\lambda\psi(x(s),y)}a(x(s), y)f (y)\,dy,\]
and denote by $\mathcal{K}(s, t)$ the kernel of operator $\mathcal{T}_{\lambda}\mathcal{T}_{\lambda}^{*}$. Then we have
$$\mathcal{K}(s, t)=\int e^{i\lambda[\psi(x(s),y)-\psi(x(t),y)]}a(x(s),y)\overline{a(x(t),y)}\,dy.$$
Our goal is reduced to proving
\begin{equation}\label{gg24}	\Big\|W\mathcal{T}_{\lambda}\mathcal{T}^{*}_{\lambda}\overline{W}\Big\|_{\mathfrak{S}^{2}(L^2(\Sigma))}\lesssim\lambda^{-1/2}(\log\lambda)^{1/2}\|W\|^2_{L^{4}(\Sigma)},   {\text{for all}}  W\in L^{4}(\Sigma),
	\end{equation}
   which is equivalent to 
    \begin{equation}\label{gg25}
    \Big(\int_0^1\int_0^1 |W(x(s))\mathcal{K}(s,t)W(x(t))|^2\,ds\,dt\Big)^{1/2}\lesssim\lambda^{-1/2}(\log\lambda)^{1/2}\|W\|^2_{L^{4}(\Sigma)}.
    \end{equation}

Since $y$ stays in the annulus with inner radius ${\epsilon_0}/3$ and outer radius ${4\epsilon_0}/3$, we can restrict $y$
to the circle with radius $r$. In fact, we can represent $y$ in polar coordinates as $y = r\omega$
(i.e., geodesic polar coordinates on $M$), $\epsilon_0/3\leqslant r\leqslant 4\epsilon_0/3$, $\omega = (\omega_1, \omega_2)\in \mathbb S^1$ and denote:
\[\psi_r(x(s), \omega) = \psi(x(s), y),  a_r(x(s), \omega)\overline{a_r(x(t),\omega)} = \kappa(r,\omega)a_r(x(s), \omega)\overline{a_r(x(t),\omega)}\] for some smooth function $\kappa$.
Define 
$$\mathcal{K}_r(s, t)=\int_{\mathbb{S}^{1}} e^{i\lambda[\psi_r(x(s),\omega)-\psi_r(x(t),\omega)]}a_r(x(s),\omega)\overline{a_r(x(t),\omega)}\,d\omega.$$
Then, \[\mathcal{K}(s,t)=\int_{\epsilon_0/3}^{4\epsilon_0/3} \mathcal{K}_r(s,t)\,dr.\]

 By applying \cite[Lemma 3.2]{bgt2007}, we have that $\mathcal K_r(s,t)$ is bounded by
$$C(1+\lambda|s-t|)^{-1/2},$$ 
and then \[\mathcal{K}(s,t)\lesssim (1+\lambda|s-t|)^{-1/2}.\]
Consequently, the square of the left-hand side of  \eqref{gg25} is bounded by
\begin{equation}\label{gg26}
    \int_0^1\int_0^1 |W(x(s))|^2(1+\lambda|s-t|)^{-1}|W(x(t))|^2\,dt\,ds,
\end{equation}

Applying the Young inequality, we obtain
\begin{equation}\label{young}
\begin{aligned}
\eqref{gg26}&\leq\|W\|^2_{L^4(\Sigma)}\cdot\Big\|\int^1_0(1+\lambda|s-t|)^{-1}|W(x(t))|^2dt\Big\|_{L^2(0,1)}\\
&\leq\|W\|_{L^4(\Sigma)}^4\cdot\int^1_0(1+\lambda t)^{-1}dt\\
&\lesssim\lambda^{-1}\log\lambda\|W\|_{L^4(\Sigma)}^4.
\end{aligned} 
\end{equation}
Taking square roots completes the proof of \eqref{gg24}. 

We now prove the logarithm-free refinement in Theorem \ref{main theorem} when $k=d-1$, beginning with the case $ 2 \leq p < 4 $.

Let $1<\alpha<2$, and thus $\alpha'>2$. the operator $W\,T_\lambda T_\lambda^{*}\,\overline{W}$
is self-adjoint.  Invoking the Hausdorff–Young theorem for integral
operators \cite[Theorem 1]{russo} and then Young’s inequality we obtain
\begin{align*}
    \Big\|W T_\lambda T_\lambda^* \overline{W}\Big\|_{\mathfrak{S}^{\alpha'}}&\leq\Big(\int^1_0\Big(\int^1_0\big|W(x(s))\mathcal{K}(s,t)\overline{W(x(t))}\big|^{\alpha}\,ds\Big)^{\alpha'/\alpha}\,dt\Big)^{1/\alpha'}\\
    &\leq \|W\|_{L^{2\alpha'}(\Sigma)}\Big(\int^1_0\Big(\int^1_0|W(x(s))|^{\alpha}(1+\lambda|s-t|)^{-\alpha/2}\,ds\Big)^{2\alpha'/\alpha}\,dt\Big)^{1/2\alpha'}\\
    &\le\|W\|^2_{L^{2\alpha'}(\Sigma)}\Big(\int^1_0(1+\lambda|s|)^{-\frac{\alpha}{2}}\,ds\Big)^{1/\alpha}\\
    &\lesssim\lambda^{-1/2}\|W\|^2_{L^{2\alpha'}(\Sigma)}.
\end{align*}

Set $ 2 p / (p - 2) = 2 \alpha' > 4 $, equivalently $ 2 \leq p < 4 $. By the duality principle (Lemma \ref{l1}), the above estimate gives us
\begin{equation}
\label{near2ineq}
\Big\| \sum_j t_j | T_\lambda f_j |^2 \Big\|_{L^{p/2}(\Sigma)} \leq C_p \lambda^{-1/2} \big\| \{t_j\} \big\|_{l^{\alpha}(J)}.
\end{equation}

Since the curve $\Sigma$ is compact, Hölder's inequality implies that for any $2 \le q \le p$, we have
\[
\Big\| \sum_j t_j | T_\lambda f_j |^2 \Big\|_{L^{q/2}(\Sigma)} \leq C_\alpha \lambda^{-1/2} \big\| \{t_j\} \big\|_{l^\alpha(J)}.
\]
Note that we can choose $p$ in \eqref{near2ineq} arbitrarily close to $4$, thus the associated exponent $\alpha$ can be taken arbitrarily close to $2$. This completes the proof for the range $2\le p<4$.

Now let $4<p\leq\infty$.
Note that for any $p_0\in(4,p)$, we have
\begin{align*}
\Big\|WT_{\lambda}T^{*}_{\lambda}\overline{W}\Big\|_{\mathfrak{S}^2(L^2(\Sigma))}&=\Big(\int^1_0\int^1_0|W(x(s))|^2|\mathcal{K}(s,t)|^2|\overline{W(x(t))}|^2\,ds\,dt\Big)^{1/2}\\
&\leq\|W\|_{L^{2p_0/(p_0-2)}(\Sigma)}\Big\|\int^1_0 |\mathcal{K}(s,t)|^2|W(x(t))|^2\,dt\Big\|^{1/2}_{L^{p_0/2}(0,1)}\\
&\leq\|(1+\lambda|s|)^{-1}\|^{1/2}_{L^{p_0/4}(0,1)}\|W\|^2_{L^{2p_0/(p_0-2)}(\Sigma)}\\
&\leq C_{p_0}\lambda^{-2/p_0}\|W\|_{L^{2p_0/(p_0-2)}(\Sigma)}^2.
\end{align*}
Interpolating this with the trace-class bound 
\eqref{g23} yields
\[
\Big\|WT_{\lambda}T^{*}_{\lambda}\overline{W}\Big\|_{\mathfrak{S}^{(2p/p_0)'}(L^2(\Sigma))}\leq C_{p_0}\lambda^{-2/p}\|W\|^2_{L^{2p/(p-2)}(\Sigma)}.
\]
By duality this is equivalent to
\[\Biggl\|\sum_{j\in J}t_j|T_{\lambda}f_j|^2\Biggr\|_{L^{p/2}(\Sigma)}\leq C_{p_0}\lambda^{-2/p}\big\|\{t_j\}\big\|_{l^{2p/p_0}(J)}. \]
Since $p_0$ can be taken to be arbitrarily close to $4$, we see that the following holds for any $1\le\alpha<p/2.$
\[\Biggl\|\sum_{j\in J}t_j|T_{\lambda}f_j|^2\Biggr\|_{L^{p/2}(\Sigma)}\leq C_\alpha\lambda^{-2/p}\big\|\{t_j\}\big\|_{l^{\alpha}(J)}, \]
as desired. This completes the proof of Theorem \ref{main theorem} when $k=d-1$ in dimension two.
\end{proof}

\section{Proof in higher dimensions}\label{sec 5}
In this section, we prove Theorem \ref{main theorem} for $d\ge3$.
Assume that in the geodesic normal coordinate system about $x_0\in M,$ $\Sigma$  is parameterized
by $x(z_1 ,z_2 ,\cdots,z_k),$ $ x(0) = 0$.
Again by a partition of unity, we can assume that (a local piece of) $\Sigma$ is contained
in a small enough geodesic ball about $x_0$, i.e., $|x(z)|\leq \epsilon_0/10,$ which forces $\epsilon_0/3\leq |y| \leq 4\epsilon_0/3$. 
 We write
 $$\mathcal{T}_{\lambda}(f)(z):=T_{\lambda}(f)(x(z))=\int_{M}e^{i\lambda\psi(x(z),y)}a_{\lambda}(x(z),y)f(y)\,dy$$
 and denote by ${K}(x,x')$ the kernel of the operator ${T}_{\lambda}{T}^*_{\lambda}$. Let
 $$\mathcal{K}(z,z'):=K(x(z),x(z'))=\int_Me^{i\lambda[\psi(x(z),y)-\psi(x(z'),y)]}a_{\lambda}(x(z),y)\overline{a_{\lambda}(x(z'),y)}\,dy$$ 
 be the kernel of the operator $\mathcal{T}_{\lambda}\mathcal{T}^*_{\lambda}$. We now turn to an estimate from \cite{bgt2007} that will be essential for our analysis.
 \begin{lemma}[Lemma 6.1 in \cite{bgt2007}]\label{l6} For any $(x,x')$,
 	$|{K}(x,x')|\lesssim 1$. 
If $|\psi(x,x')|\gtrsim\lambda^{-1}$, then for  any $N\in\mathbb{Z}^+$, there exists $\epsilon_0\ll1$ so that 
 		\begin{equation}\label{g17}
 		{K}(x,x')=\sum_{\pm}\sum^{N-1}_{n=0}\frac{e^{\pm i\lambda \psi(x,x')}}{(\lambda \psi(x,x'))^{\frac{d-1}{2}+n}}a_n^{\pm}(x,x',\lambda)+b_N(x,x',\lambda),
 		\end{equation}
 	where $a^{\pm}_n,b_n\in C^{\infty}(\mathbb{R}^d\times\mathbb{R}^d\times\mathbb{R})$ and $-\psi(x,y)=d_g(x,y)$ is the geodesic distance
     between $x$ and $x'$.  Furthermore,  $a^{\pm}$ are real, have supports of size $O(\epsilon_0)$ with respect to the first two variables and are uniformly
 	bounded with respect to $\lambda$. Finally
 	\begin{equation}\label{gb}   
 	|b_N(x,x',\lambda)|\lesssim |1+\lambda \psi(x-x')|^{-\frac{d-1}{2}-N}.
  \end{equation}
 \end{lemma}
In view of $$d_g(x(z),x(z'))\sim|z-z'|,$$
 noting that $K$ is bounded, applying Lemma \ref{l6} yields a rough bound on the kernel of $\mathcal{T}_{\lambda}\mathcal{T}^*_{\lambda}$
 \begin{equation}\label{g9}
 |\mathcal{K}(z,z')|\leq C(1+\lambda|z-z'|)^{-\frac{d-1}{2}}.
 \end{equation}
However, this estimate alone is insufficient to achieve our desired bounds. To improve upon it, we follow the approach in \cite{bgt2007} by exploiting the oscillatory nature of the phase in $K(x(z),x(z'))$. This involves dyadically decomposing the kernel based on the size of $|z-z'|$. We now provide a detailed description of this decomposition and collect the key estimates that will be instrumental in our proof.

 We fix a compactly supported bump function $\chi_0\in C^{\infty}_0(\mathbb{R}^k),$ so that $ \mathrm{supp} \,\chi_0\in
\{x\in\mathbb{R}^k:|x|\leq C\}$.
Additionally, let  $\tilde{\chi}\in C^{\infty}_0(\mathbb{R}^k)$  be supported in the set $\{x\in{\mathbb{R}^k}:\frac{1 }{2}<|x|<2\}$ such that there is a partition of unity on $\{x\in\mathbb{R}^k: |x|<1\}$ of the form
\begin{equation}\label{g10}
1=\chi_0(\lambda x)+\sum^{\log\lambda/\log2}_{j=1}\tilde{\chi}(2^jx).
\end{equation}
Using this, we decompose the kernel  $\mathcal{K}(z,z')$ dyadically:
\begin{equation}\label{dyadic}
\begin{aligned}
\mathcal{K}(z,z')&=\mathcal{K}(z,z')\chi_0(\lambda(z-z'))+\sum^{\log\lambda/\log2}_{j=1}\mathcal{K}(z,z')\tilde{\chi}(2^j(z-z'))\\
&= \mathcal{K}_0(z,z')+\sum^{\log\lambda/\log2}_{j=1}\mathcal{K}_j(z,z').
\end{aligned}
\end{equation}
In the proof of Theorem \ref{main theorem} when $k\le d-2$, we will estimate the contribution of each dyadic piece $\mathcal K_j$. To facilitate these estimates, we rely on the following crucial estimates from \cite{bgt2007}.
\begin{lemma}[Proposition 6.3 in \cite{bgt2007}]\label{l7}
For sufficiently large $j$, let $(\mathcal{T}_{\lambda}\mathcal{T}_{\lambda}^*)_j$ be the operator whose integral kernel is $\mathcal{K}_j(z,z')$. Then it satisfies the estimate
	\begin{align}
	&\|(\mathcal{T}_{\lambda}\mathcal{T}^*_{\lambda})_j f\|_{L^{\infty}(\Sigma)}\lesssim\Big(\frac{2^j}{\lambda}\Big)^{\frac{d-1}{2}}\|f\|_{L^1(\Sigma)},\label{g26}\\
	&\|(\mathcal{T}_{\lambda}\mathcal{T}^*_{\lambda})_j f\|_{L^2(\Sigma)}\lesssim2^{-j k}\Big(\frac{2^j}{\lambda}\Big)^{\frac{d-1}{2}+\frac{k-1}{2}}\|f\|_{L^2(\Sigma)}.\label{g27}
	\end{align}
\end{lemma}
We remark that the operator corresponding to $\mathcal K_0$ satisfies \eqref{g26} and \eqref{g27} with $2^j\sim\lambda,$ and thus we do not need to handle it separately. 
    Furthermore, note that the operator $(\mathcal{T}_\lambda\mathcal{T}^*_\lambda)_j$ in Proposition 6.3 of \cite{bgt2007} does not include the remainder term $b_N$. However, one readily verifies that the same estimates hold when $b_N$ is included, provided $N$ is chosen sufficiently large. Indeed, by \eqref{gb}, if $N$ in Lemma \ref{l6} is sufficiently large, then the operator with kernel
\[
b_N(z,z',\lambda)\,\tilde\chi\bigl(2^j(z-z')\bigr)
\]
satisfies \eqref{g26} and \eqref{g27} via Young’s inequality.
\subsection{Proof of Theorem \ref{main theorem} when $k\le d-2$}

We begin by adopting the same strategy as in the two-dimensional case.
By the reasoning leading to \eqref{g24}, it suffices to prove that, for every
$W\in L^{2p/(p-2)}(\Sigma)$,
\begin{equation}\label{g37}
\big\|W\mathcal{T}_\lambda\mathcal{T}_\lambda^*\overline{W}\big\|_
{\mathfrak{S}^{(p/2)'}(L^2(\Sigma))}
\lesssim
(\log\lambda)^{h(k,d,p)}
\lambda^{2\delta(k,d,p)-(d-1)}
\|W\|_{L^{2p/(p-2)}(\Sigma)}^2.
\end{equation}

We shall apply the dyadic decomposition \eqref{dyadic} to the kernel of $\mathcal{T}_\lambda \mathcal{T}_\lambda^*$. This decomposition enables us to leverage Proposition \ref{l7} to obtain bounds on the Schatten norms of the operators $ W (\mathcal{T}_\lambda \mathcal{T}_\lambda^*)_j \overline{W} $. Summing these bounds over all dyadic levels $ j $ then yields the desired Schatten bound for $ W \mathcal{T}_\lambda \mathcal{T}_\lambda^* \overline{W} $.
	
 By Lemmas \ref{z1} and using the fact that the largest singular value of a compact operator is equal to its operator norm, we have
	\begin{equation}
	\begin{aligned}
\Big\|W(\mathcal{T}_{\lambda}\mathcal{T}^*_{\lambda})_j\overline{W}\Big\|_{\mathfrak{S}^{\infty}(L^2(\Sigma))}
	&\leq \|W\|^2_{L^{\infty}(\Sigma)}\|(\mathcal{T}_{\lambda} \mathcal{T}^*_{\lambda})_j\|_{\mathfrak{S}^{\infty}(L^2(\Sigma))}\\
 &=\|W\|^2_{L^{\infty}(\Sigma)}\|(\mathcal{T}_{\lambda} \mathcal{T}^*_{\lambda})_j\|_{L^2(\Sigma)\to L^2(\Sigma)}.\\
	\end{aligned}
	\end{equation}
Inserting the operator‐norm bound $\eqref{g27}$ into the above inequality gives
\begin{equation}\label{g15}
\Big\|W(\mathcal{T}_{\lambda}\mathcal{T}^*_{\lambda})_j\overline{W}\Big\|_{\mathfrak{S}^{\infty}(L^2(\Sigma))}\leq C_1 2^{-jk}\Big(\frac{2^j}{\lambda}\Big)^{\frac{d-1}{2}+\frac{k-1}{2}}\|W\|^2_{L^{\infty}(\Sigma)},
\end{equation}
where $C_1$ is a constant that
depends only on the fixed data $(M, \Sigma)$.

 Recalling that $\tilde{\chi}$ is supported in the set $\{x\in{\mathbb{R}^k}:\frac{1}{2}<|x|<2\}$ and applying the kernel estimate
\eqref{g9}, or directly using the bound $\eqref{g26}$, one shows
 \begin{equation}\label{q1}
     \begin{aligned}
\Big\|W(\mathcal{T}_{\lambda}\mathcal{T}^*_{\lambda})_j\overline{W}\Big\|_{\mathfrak{S}^{2}(L^2(\Sigma))}&=\Big(\iint_{\Sigma\times\Sigma} \Big|W\mathcal{K}_j(\cdot,\cdot)\overline{W}\Big|^2\Big)^{1/2}\\
         &\leq C_2 \Big(\int_{B^{k}}\int_{B^k }\Big|W(x(z))\Big|^2(1+2^{-j}\lambda)^{-(d-1)}\Big|\overline{W(x(z'))}\Big|^2dzdz'\Big)^{1/2}\\
         &\leq C_2 \Big(\frac{2^j}{\lambda}\Big)^{\frac{d-1}{2}}\|W\|^2_{L^2(\Sigma)},
     \end{aligned}
 \end{equation}
where $B^k$ denotes the unit ball in $\mathbb{R}^k$ and $C_2$ is an absolute constant. Here we used the fact that for every operator $T$ acting on functions on $ L^{2}(\Sigma) $,
$$\big\|T\big\|^2_{\mathfrak{S}^2(L^2(\Sigma))}=\int_{\Sigma}\int_{\Sigma}|T(x,y)|^2\,dx\,dy,$$
where $ T(\cdot, \cdot) $ denotes the integral kernel of  $ T $.

 Interpolating between the bounds \eqref{g15} and \eqref{q1} in the Schatten spaces shows that for any $p_0\geq2$,
\begin{equation}
\Big\|W(\mathcal{T}_{\lambda}\mathcal{T}^*_{\lambda})_j\overline{W}\Big\|_{\mathfrak{S}^{2p_0/(p_0-2)}(L^2(\Sigma))}\leq C_{p_0}
2^{-\frac{2jk}{p_0}}\Big(\frac{2^j}{\lambda}\Big)^{\frac{d-1}{2}+\frac{k-1}{p_0}}\|W\|^2_{L^{2p_0/(p_0-2)}(\Sigma)},
\end{equation}
where we take $\frac{2p_0}{p_0-2}=\infty$ as $p_0=2$ and $\frac{2p_0}{p_0-2}=2$ as $p_0=\infty$.
Summing over $1\ll j\leq\log{\lambda}/\log{2}$ yields\footnote{Here we may require $j$ to be sufficiently large by choosing $\epsilon_0$ small.}
\begin{equation}\label{nf2}
\begin{aligned}
\Big\|W\mathcal{T}_{\lambda}\mathcal{T}^*_{\lambda}\overline{W}\Big\|&_{\mathfrak{S}^{2p_0/(p_0-2)}(L^2(\Sigma))}
\leq C_{p_0}\lambda^{-\frac{d-1}{2}-\frac{k-1}{p_0}}\sum_{j=1}^{\log\lambda/\log2}2^{j(\frac{d-1}{2}-\frac{k+1}{p_0})}\|W\|^2_{L^{2p_0/(p_0-2)}(\Sigma)}\\
&\leq C_{p_0}\begin{cases}
\lambda^{-\frac{d-1}{2}-\frac{k-1}{p_0}+\frac{d-1}{2}-\frac{k+1}{p_0}}\|W\|^2_{L^{2p_0/(p_0-2)}(\Sigma)},\quad
	\text{if}\quad p_0>\frac{2(k+1)}{d-1},\\
	\lambda^{-\frac{d-1}{2}-\frac{k-1}{p_0}}\|W\|^2_{L^{2p_0/(p_0-2)}(\Sigma)},\qquad\qquad\quad\,\text{if}\quad p_0<\frac{2(k+1)}{d-1},\\
	\lambda^{-\frac{d-1}{2}-\frac{k-1}{p_0}}\log\lambda\|W\|^2_{L^{2p_0/(p_0-2)}(\Sigma)},\qquad\quad\text{if}\quad p_0=\frac{2(k+1)}{d-1}.
\end{cases}
\end{aligned}
\end{equation}
Note that estimate \eqref{nf2} remains valid for any $\Sigma$ with $1 \le \dim \Sigma \le d-1$. In particular, we will later apply \eqref{nf2} in the proof of Theorem \ref{main theorem}.

Now we first consider the case where $\dim \Sigma < d - 2$. A direct calculation yields
\begin{equation}\label{n20}
\begin{aligned}
\Big\|W\mathcal{T}_{\lambda}\mathcal{T}^*_{\lambda}\overline{W}\Big\|_{\mathfrak{S}^{1}(L^2(\Sigma))}&=\|W\mathcal{T}_{\lambda} \|^2_{\mathfrak{S}^{2}(L^2(M)\to L^2(\Sigma))}\\
&=\int_{\Sigma}|W(x)|^2\int_M| \mathcal{T}_{\lambda}(x,y)|^2dx'dz\\
&\leq C_3\|W\|^2_{L^{2}(\Sigma)},
\end{aligned}
\end{equation}
where $\mathcal{T}_\lambda(x,y) = e^{-i \lambda d_g(x(z), y)} a_\lambda(x,y)$ is the kernel of $\mathcal{T}_\lambda$ and $C_3$ is a constant that
depends only on the fixed data $(M, \Sigma)$.

For $p_0 > \frac{2(k+1)}{d-1}$, interpolating this estimate with the previous Schatten bounds \eqref{nf2} yields
\begin{equation}\label{nf1}
\Big\| W \mathcal{T}_\lambda \mathcal{T}_\lambda^* \overline{W} \Big\|_{\mathfrak{S}^{\frac{2p}{2p - p_0 - 2}}(L^2(\Sigma))} \leq C_{p_0}\lambda^{-\frac{2k}{p}} \|W\|_{L^{2p/(p-2)}(\Sigma)}^2.
\end{equation}

In particular, choosing $p_0 = 2$, we obtain
\[
\Big\| W \mathcal{T}_\lambda \mathcal{T}_\lambda^* \overline{W} \Big\|_{\mathfrak{S}^{\frac{p}{p-2}}(L^2(\Sigma))} \leq C_1^{\frac2p}C_3^{1-\frac{2}p} \lambda^{-\frac{2k}{p}} \|W\|_{L^{2p/(p-2)}(\Sigma)}^2,
\]
which establishes the desired bound $\eqref{g37}$ for $\dim \Sigma = k < d - 2$. It remains to consider the case $k=d-2$. Setting $p_0=2$ in
\eqref{nf2} gives
\[
\Big\|W\mathcal{T}_\lambda\mathcal{T}_\lambda^*\overline{W}\Big\|_
{\mathfrak{S}^\infty(L^2(\Sigma))}
\lesssim
\lambda^{-(d-2)}\log\lambda\,
\|W\|_{L^\infty(\Sigma)}^2.
\]
Interpolating this estimate with the trace-class bound \eqref{n20}, we
obtain, for every $2\leq p\leq\infty$,
\[
\Big\|W\mathcal{T}_\lambda\mathcal{T}_\lambda^*\overline{W}\Big\|_
{\mathfrak{S}^{p/(p-2)}(L^2(\Sigma))}
\lesssim
\lambda^{-\frac{2(d-2)}{p}}
(\log\lambda)^{\frac{2}{p}}
\|W\|_{L^{2p/(p-2)}(\Sigma)}^2.
\]
Since $k=d-2$, this is precisely \eqref{g37}, and hence proves the
endpoint estimate in Theorem~\ref{main theorem}.

We next prove the logarithm-free refinement. Let $p>2$ and choose
$2<p_0\leq p$. Since
\[
p_0>\frac{2(k+1)}{d-1}=2,
\]
estimate \eqref{nf1} and the duality principle yield
\[
\Biggl\|\sum_{j\in J}t_j|T_\lambda f_j|^2\Biggr\|_{L^{p/2}(\Sigma)}
\le C_{p_0}
\lambda^{-\frac{2(d-2)}{p}}
\|\{t_j\}\|_{l^{\frac{2p}{p_0+2}}(J)}.
\]
Given any $\alpha<p/2$, we may choose $p_0>2$ sufficiently close to $2$
so that
\[
\alpha<\frac{2p}{p_0+2}<\frac{p}{2}.
\]
By the monotonicity of sequence norms,
\[
\|\{t_j\}\|_{l^{\frac{2p}{p_0+2}}(J)}
\leq
\|\{t_j\}\|_{l^\alpha(J)},
\]
and hence
\[
\Biggl\|\sum_{j\in J}t_j|T_\lambda f_j|^2\Biggr\|_{L^{p/2}(\Sigma)}
\le C_\alpha
\lambda^{-\frac{2(d-2)}{p}}
\|\{t_j\}\|_{l^\alpha(J)}.
\]
This completes the proof of Theorem~\ref{main theorem} in the case
$k=d-2$.\qed
 \subsection{Proof of Theorem \ref{main theorem} when $k=d-1$ and $d\ge3$}

We begin by deriving \eqref{main estimate} in this case. We observe that inequality \eqref{nf2} covers \eqref{main estimate} when $2 \leq p \leq \frac{2d}{d-1}$. Therefore, it suffices to consider the case $\frac{2d}{d-1} < p \leq \infty$.

Now we consider the range $4 \leq p \leq \infty$. Recall that $\alpha(d-1,d,p) = \frac{p}{2}$. In particular, when $p = 4$, we have $\alpha(d-1,d,4) = 2$. Applying inequality \eqref{g9} and Young's inequality, we obtain
\begin{equation}\label{f11}
\begin{aligned}
\Big\| W \mathcal{T}_{\lambda} \mathcal{T}_{\lambda}^* \overline{W} &\Big\|_{\mathfrak{S}^2(L^2(\Sigma))} 
= \left( \iint_{B^{d-1} \times B^{d-1}} \left| W(x(z)) K(x(z), x(z')) \overline{W(x(z'))} \right|^2 \, dz\, dz' \right)^{1/2} \\
&\lesssim \left( \iint_{B^{d-1} \times B^{d-1}} |W(x(z))|^2 (1 + \lambda |z - z'|)^{-(d-1)} |\overline{W(x(z'))}|^2 \, dz\, dz' \right)^{1/2} \\
&\lesssim \|W\|_{L^4(\Sigma)} \left\| \int_{B^{d-1}} (1 + \lambda |z - z'|)^{-(d-1)} |\overline{W(x(z'))}|^2 \, dz' \right\|_{L^2(\Sigma)}^{1/2} \\
&\leq \|W\|_{L^4(\Sigma)}^2 \left| \int_{B^{d-1}} (1 + \lambda |z|)^{-(d-1)} \, dz \right|^{1/2} \\
&\lesssim \lambda^{-(d-1)/2} (\log \lambda)^{1/2} \|W\|_{L^4(\Sigma)}^2.
\end{aligned}
\end{equation}

Interpolating this estimate with the Schatten bounds given in \eqref{n20} yields
\[
\Big\| W \mathcal{T}_{\lambda} \mathcal{T}_{\lambda}^* \overline{W} \Big\|_{\mathfrak{S}^{p/(p-2)}(L^2(\Sigma))}
\lesssim \lambda^{-2(d-1)/p} (\log \lambda)^{2/p} \|W\|^2_{L^{2p/(p-2)}(\Sigma)}.
\]
This completes the proof of \eqref{main estimate} for the case $4 \leq p \leq \infty$.

Next, consider the range $\frac{2d}{d-1} < p < 4$. By interpolating the bound in \eqref{f11} with inequality \eqref{nf2}—which applies when $\dim \Sigma = d-1$ and $p_0 = \frac{2d}{d-1}$—we obtain
\[
\Big\| W \mathcal{T}_{\lambda} \mathcal{T}_{\lambda}^* \overline{W} \Big\|_{\mathfrak{S}^{\frac{2 p(d-2)}{2 p d - 4 d - 3 p + 4}}(L^2(\Sigma))}
\lesssim \lambda^{-2(d-1)/p} (\log \lambda)^{\frac{2 d - p}{p(d-2)}} \|W\|^2_{L^{\frac{2 p}{p-2}}(\Sigma)}.
\]
This completes the analysis for $\frac{2 d}{d-1} < p < 4$.

Finally, we turn to the logarithm-free refinement in Theorem \ref{main theorem} when $k=d-1$. The approach follows a similar pattern to the two-dimensional case. For any $p_0 \in (4, p)$, applying the bound \eqref{g9} and Young's inequality gives
\begin{align*}
\Big\| W \mathcal{T}_{\lambda} \mathcal{T}_{\lambda}^* \overline{W} &\Big\|_{\mathfrak{S}^2(L^2(\Sigma))}
= \left( \iint_{B^{d-1} \times B^{d-1}} |W(x(z))|^2 |K(x(z), x(z'))|^2 |\overline{W(x(z'))}|^2 \, dz\, dz' \right)^{1/2} \\
&\leq \|W\|_{L^{2p_0/(p_0-2)}(\Sigma)} \left\| \int_{B^{d-1}} |K(x(z), x(z'))|^2 |W(x(z'))|^2 \, dz' \right\|_{L^{p_0/2}(B^{d-1})}^{1/2} \\
&\leq \left\| (1 + \lambda |z|)^{-(d-1)} \right\|_{L^{p_0/4}(B^{d-1})}^{1/2} \|W\|_{L^{2p_0/(p_0-2)}(\Sigma)}^2 \\
&\leq C_{p_0} \lambda^{-2(d-1)/p_0} \|W\|_{L^{2p_0/(p_0-2)}(\Sigma)}^2.
\end{align*}
Interpolating this with the trace-class bound in \eqref{n20} yields
\[
\Big\| W \mathcal{T}_{\lambda} \mathcal{T}_{\lambda}^* \overline{W} \Big\|_{\mathfrak{S}^{(2p/p_0)'}(L^2(\Sigma))} \leq C_{p_0} \lambda^{-2(d-1)/p} \|W\|_{L^{2p/(p-2)}(\Sigma)}^2.
\]
 By the duality principle (Lemma \ref{l1}), this is equivalent to
\[
\Biggl\| \sum_{j \in J} t_j |T_{\lambda} f_j|^2 \Biggr\|_{L^{p/2}(\Sigma)} 
\leq C_{p_0} \lambda^{-2(d-1)/p} \left\| \{t_j\} \right\|_{l^{2p/p_0}(J)}.
\]
Since $p_0$ can be chosen arbitrarily close to 4, it follows that for any $\alpha < p/2$,
\[
\Biggl\| \sum_{j \in J} t_j |T_{\lambda} f_j|^2 \Biggr\|_{L^{p/2}(\Sigma)} 
\leq C_\alpha \lambda^{-2(d-1)/p} \left\| \{t_j\} \right\|_{l^\alpha(J)}.
\]
This completes the proof of Theorem \ref{main theorem} when $k=d-1$ in dimensions $d\ge3$.
\qed

\section{Sharpness in dimension two}\label{sec 6}
 In this section, we demonstrate that the $d=2$ case of Theorem \ref{main theorem} is essentially sharp for any intermediate value of $\#J$, except for the logarithmic loss on the two-sphere $\mathbb S^2$. 
 To achieve this, we will utilize a key estimate from \cite{frank2017}.

  When $M$ is the standard sphere $\mathbb{S}^2$ and $2\leq p\leq\infty$,
  consider an increasing function, $t_{\lambda}=t(\lambda),$ $\lambda>0$, satisfying the following conditions:
 $$\lim_{\lambda\to\infty}t_{\lambda}=+\infty,      \lim_{\lambda\to\infty}\dfrac{t_{\lambda}}{\lambda}=0.$$
Our goal is to construct, for sufficiently large $\lambda$, geodesic segments  $\gamma\subset\mathbb{S}^2$ and orthonormal systems $\{f_j\}_{j\in J_{\lambda}}\subset E_{\lambda}$ and $ \{g_j\}_{j\in J_{\lambda}}\subset E_{\lambda}$ with $\#J_{\lambda}\sim t_{\lambda}$, such that the following estimates hold:
 \begin{align}
 &\Big\|\sum_{j\in J_{\lambda}}|f_j|^2\Big\|_{L^{p/2}(\gamma)}\gtrsim\lambda^{1-\frac{2}{p}}t_{\lambda}^{\frac{2}{p}};\label{o1}\\
 &\Big\|\sum_{j\in J_{\lambda}}|g_j|^2\Big\|_{L^{p/2}(\gamma)}\gtrsim\lambda^{\frac{1}{2}}t^{\frac{1}{2}}_{\lambda}.\label{o2}
 \end{align}
 For each $\lambda > 0$, we construct orthonormal systems using spherical harmonics of degree $k \sim \lambda$. Recall that for any spherical harmonic $Y(x)$ of degree $k$ on $L^2(\mathbb{S}^2)$, the Laplace--Beltrami operator satisfies
\[
\Delta_{\mathbb{S}^2} Y(x) = -k(k+1) Y(x).
\]
Let $Y_k^{\alpha_k}$ denote an $L^2$-normalized spherical harmonic of degree $k$, where $\alpha_k \in [-k, k]$. Consider a sequence $\{t_k\}_{k\in\mathbb{N}}$ such that
\[
\lim_{k \to \infty} t_k = +\infty, \quad \text{and} \quad \lim_{k \to \infty} \frac{t_k}{k} = 0.
\]
Assuming that for sufficiently large $k$, $t_k \leq \frac{k}{2}$, we claim the following:
\begin{itemize}
 \item  The orthonormal system $\{Y_k^{\alpha_k}\}_{\alpha_k \in [t_k, 2t_k] \cap \mathbb{Z}}$ satisfies the bound corresponding to \eqref{o1}.
\item  The orthonormal system $\{Y_k^{\alpha_k}\}_{\alpha_k \in [k - 2t_k,\, k - t_k] \cap \mathbb{Z}}$ satisfies the bound corresponding to \eqref{o2}.
\end{itemize}
To verify these claims, we will utilize a key estimate from \cite{frank2017}, which provides essential bounds for these spherical harmonics.
\begin{lemma}[Proposition 15 in \cite{frank2017}]\label{l14}
If we parameterize points in $\mathbb{S}^2$ by usual spherical coordinates
$(\theta,\varphi)\in[0,\pi]\times[0,2\pi)
$, then 
\begin{enumerate}[(1)]
    \item 
there exist $\eta_1>0, K\geq1,$ and $c>0$
such that for all $k\geq K$, $\eta_1t_k/k\leq\theta\leq\frac{\pi}{2}$ and  $0\leq\varphi\leq2\pi$, 
$$\sum_{\alpha_k\in [t_k,2t_k]\cap\mathbb{Z}}|Y_k^{\alpha_k}(\theta,\varphi)|^2\geq c\,\frac{t_k}{\sin \theta};$$ 
\item there exist $\eta_2>0, K\geq 1$, and $c>0$ such that for all $k\geq K$, $0\leq\theta\leq\eta_2(t_k/k)^{1/2}$ and  $0\leq\varphi\leq2\pi$,
$$\sum_{\alpha_k\in [k-2t_k,k-t_k]\cap\mathbb{Z}}|Y_k^{\alpha_k}(\pi/2-\theta,\varphi)|^2\geq c\,k^{1/2}t_k^{1/2}.$$
\end{enumerate}
\end{lemma}
Now, we are prepared to proceed with the proof of our sharpness result.
\begin{theorem}
    There exist $c>0 $, $K\geq1$ and curves $\gamma_1$ and $\gamma_2$ on $\mathbb{S}^2$ so that for all $k\geq K$
    and $2\leq p\leq\infty$ one has 
   \begin{align}
    &\Big\|\sum_{\alpha_k\in [t_k,2t_k]\cap\mathbb{Z}}|Y_k^{\alpha_k}(\theta,\varphi)|^2\Big\|_{L^{p/2}(\gamma_1)}\geq Ck^{1-\frac{2}{p}}t^{\frac{2}{p}}_k;\label{s1}\\
    &\Big\|\sum_{\alpha_k\in [k-2t_k,k-t_k]\cap\mathbb{Z}}|Y_k^{\alpha_k}(\theta,\varphi)|^2\Big\|_{L^{p/2}(\gamma_2)}\geq Ck^{\frac{1}{2}}t^{\frac{1}{2}}_k.\label{s2}
    \end{align}
    In particular, the $d=2$ case of Theorem \ref{main theorem} is essentially saturated by \eqref{s1} for $4\leq p\leq\infty$, and by $\eqref{s2}$ for $2\leq p\leq 4$.
\end{theorem}
\begin{proof}
   We first establish inequality \eqref{s1}. Consider the curve $\gamma_1$ on $\mathbb{S}^2$ parameterized by $\gamma_1(\theta)=(\sin\theta, 0, \cos\theta)$, with $\theta\in[0,\pi/2].$ Applying item (1) of Lemma \ref{l14}, we obtain
    \begin{equation*}
        \begin{aligned}
\Big\|\sum_{\alpha_k\in [t_k,2t_k]\cap\mathbb{Z}}|Y_k^{\alpha_k}(\theta,\varphi)|^2\Big\|^{p/2}_{L^{p/2}(\gamma_1)}&\ge\int^{\pi/2}_{\eta_1t_k/k}\Big(\sum_{\alpha_k\in [t_k,2t_k]\cap\mathbb{Z}}|Y_k^{\alpha_k}(\theta,0)|^2\Big)^{p/2}\,d\theta\\
&\gtrsim \int^{\pi/2}_{\eta_1t_k/k}\Big(\frac{t_k}{\sin\theta}\Big)^{\frac{p}{2}}\,d\theta\gtrsim  t^{p/2}_k\int^{2\eta_1t_k/k}_{\eta_1t_k/k}{\theta}^{-\frac{p}{2}}\,d\theta\\
&\gtrsim   k^{\frac p2-1}t_k.\\
        \end{aligned}
    \end{equation*}
Next, to prove \eqref{s2}, consider the curve $\gamma_2$ as the equator, that is, $\theta=\pi/2$. Applying item (2) of Lemma \ref{l14}, we conclude that
\begin{align*}
    \Big\|\sum_{\alpha_k\in [k-2t_k,k-t_k]\cap\mathbb{Z}}|Y_k^{\alpha_k}(\theta,\varphi)|^2\Big\|^{p/2}_{L^{p/2}(\gamma_2)}&=\int^{2\pi}_0\Big(\sum_{\alpha_k\in [k-2t_k,k-t_k]\cap\mathbb{Z}}|Y_k^{\alpha_k}(\pi/2,\varphi)|^2\Big)^{p/2}d\varphi\\
    &\geq \int^{2\pi}_0\big(Ck^{\frac{1}{2}}t_k^{\frac{1}{2}}\big)^{p/2}\,d\varphi\\
    &\gtrsim (k^{1/2}t_k^\frac{1}{2})^{p/2}.
\end{align*}
\end{proof}

\section{Sharpness in higher dimensions}\label{sec 7}

Finally, we address the sharpness in higher dimensions. Suppose $d\ge3$ and $\dim\Sigma = k$. We work on the standard $d$-sphere. Let 
$$
\lambda^2 = l(l+d-1),\quad l\in\mathbb{N},
$$
and fix a real orthonormal basis $\{Y_i\}_{i=1}^{\dim E_\lambda}$ of the eigenspace $E_\lambda$. Define
$$
Z_\lambda(p,x) =\sum_{i=1}^{\dim E_\lambda} Y_i(p)\,Y_i(x).
$$
By symmetry, $Z_\lambda(p,p)=C(\lambda)\sim\lambda^{d-1}$. For any $p\in\mathbb{S}^d$, the normalized zonal eigenfunction concentrated at $p$ is
$$
Z_\lambda^p(x)=\frac{Z_\lambda(p,x)}{\sqrt{Z_\lambda(p,p)}}.
$$

\begin{prop}
For $p\in\mathbb S^d$,
\[
|Z_\lambda^p(x)|
\lesssim
\begin{cases}
\lambda^{\frac{d-1}{2}},
& \operatorname{dist}(x,p)\leq\lambda^{-1},\\[4pt]
\operatorname{dist}(x,p)^{-\frac{d-1}{2}},
& \lambda^{-1}\leq\operatorname{dist}(x,p)\leq\frac{\pi}{2}.
\end{cases}
\]
Moreover,
\[
|Z_\lambda^p(x)|
\gtrsim
\lambda^{\frac{d-1}{2}}
\qquad
\text{if }
\operatorname{dist}(x,p)\leq c\lambda^{-1}
\]
for some sufficiently small constant $c>0$.
\end{prop}

\begin{prop}
For any $p_1,p_2\in\mathbb S^d$,
$$
\bigl\langle Z_\lambda^{p_1},Z_\lambda^{p_2}\bigr\rangle
=\frac{Z_\lambda(p_1,p_2)}{C(\lambda)}.
$$
\end{prop}

The proofs of these propositions can be found in \cite[Lemma 1.2.5]{DF2013}.

Let $0<\beta<1$ be a parameter to be chosen later. Fix a sufficiently large constant $C>0$, and select a $C\lambda^{-1+\beta}$–separated set $\{p_j\}_{j\in J}\subset\Sigma$ so that for all $i\neq j$,
\[
C\,\lambda^{-1+\beta}\le\mathrm{dist}(p_i,p_j)\le\frac\pi2.
\]
In this configuration,
\[
\#J\sim C^{-k}\,\lambda^{\,k(1-\beta)}.
\]
Moreover, by the preceding propositions, for $i\neq j$ we have
\[
\bigl|\langle Z_\lambda^{p_i},\,Z_\lambda^{p_j}\rangle\bigr|
\lesssim\lambda^{-\frac{d-1}{2}\,\beta}.
\]
We will need the following lemma, which follows from Gerschgorin’s Circle Theorem. For completeness, we include a proof here.
\begin{lemma}\label{l17}
Let $H$ be a Hilbert space, and let $\{e_i\}_{i=1}^n\subset H$ be unit vectors satisfying
\[
|\langle e_i, e_j\rangle|<\frac{\delta}{n-1},\quad i\neq j,
\]
for some $0<\delta<\tfrac12$. Then there exists an $n\times n$ matrix $A=[A_{ij}]$ with
\[
|A_{ij}-\delta_{ij}|\le\delta\quad(1\le i,j\le n),
\]
such that the vectors
\[
f_i=\sum_{j=1}^nA_{ij}e_j,\quad i=1,\dots,n,
\]
form an orthonormal system in $H$.
\end{lemma}

\begin{proof}
Let $G = [\langle e_i,e_j\rangle]_{i,j=1}^n$ be the Gram matrix.  By hypothesis,
\[
G_{ii} = 1,\qquad |G_{ij}| < \frac{\delta}{n-1}\quad(i\neq j),
\]
with $0<\delta<\tfrac12$. 
Write $G=I+E$ with
\[
E_{ij} = \begin{cases}\langle e_i,e_j\rangle,&i\neq j,\\0,&i=j.\end{cases}
\]
Then for  each $i$, the row–sum norm of the off–diagonal part
satisfies
\[
\|E\|_\infty =\max_i\sum_{j\neq i}|E_{ij}|
< (n-1)\cdot\frac{\delta}{n-1} = \delta < 1.
\]
  By Gerschgorin’s theorem, each eigenvalue
$\lambda_i$ of $G=I+E$ lies in $\{\,z:|z-1|\le\delta\}$, so $G$ is positive‐definite
and invertible. Since the spectral norm $\|E\|_2\le\|E\|_\infty<1$, the Neumann series for
$(I+E)^{-1/2}$ converges in operator norm.
Hence we may set
\[
A =G^{-1/2} =(I+E)^{-1/2}.
\]

Define
\[
f_i =\sum_{j=1}^n A_{ij}\,e_j.
\]
Then the Gram matrix of $\{f_i\}$ is
\[
\langle f_i,f_j\rangle =(A\,G\,A^*)_{ij}
=[G^{-1/2}\,G\,G^{-1/2}]_{ij}
=\delta_{ij},
\]
so $\{f_i\}$ is orthonormal.

Finally, we have for all $i,j$
\[
\max_{i,j}|a_{ij}-\delta_{ij}| \le \max_{1\le i\le n}{|\lambda_i^{-\frac12}-1|} \le \delta.
\]
This completes the proof.
\end{proof}

Let us return to the example.  For simplicity, let $Z_j = Z_\lambda^{p_j}$.  By the preceding lemma with  $\delta\sim (\#J)\,\lambda^{-\frac{d-1}{2}\beta}$, there exists a matrix $A = [a_{ij}]$ such that
\[
f_i = \sum_{j\in J} a_{ij}\,Z_j,
\qquad
\bigl|a_{ij}-\delta_{ij}\bigr|\lesssim (\#J)\,\lambda^{-\frac{d-1}{2}\beta},
\]
and the family $\{f_j\}_{j\in J}$ is orthonormal.

Our goal is to choose $\beta$ so that $| f_i(x)|\sim \lambda^{\frac{d-1}{2}}$ whenever $x$ lies in the $\lambda^{-1}$-neighborhood of $p_i$.  To ensure this, note that for such $x$,
\[
\Bigl|\sum_{j\neq i}a_{ij}\,Z_j(x)\Bigr|
\le
(\#J)^2\lambda^{\frac{d-1}{2}(1-\beta)}\lambda^{-\frac{d-1}{2}\beta}
\ll
\lambda^{\frac{d-1}{2}}.
\]
This holds provided
\[
(\#J)\lambda^{-\frac{d-1}{2}\beta}
=O\bigl(C^{-k}\,\lambda^{(1-\beta)k-\frac{d-1}{2}\beta}\bigr)
\ll1,
\]
which is achieved by taking
\begin{equation}\label{lower beta}
  \frac{2k}{2k + d - 1}\leq\beta<1,  
\end{equation}
and $C$ sufficiently large.
For the system $\{f_j\}$ constructed above, we then have
\[
|f_j(x)| \ge \lambda^{\frac{d-1}{2}}
\quad\text{on }B_{p_j}(\lambda^{-1}),
\]
and $\mathrm{dist}(p_i,p_j)\gg \lambda^{-1}$ for $i\neq j$.  Hence, for $p\ge2$, $\dim\Sigma\le d-1$ and $t_j\ge0$,
\[
\begin{aligned}
\Bigl\|\sum_{j\in J}t_j\,|f_j|^2\Bigr\|_{L^{p/2}(\Sigma)}
&\ge
\Bigl(\sum_{j\in J}\int_{B_{p_j}(\lambda^{-1})\cap\Sigma}
|t_j|^{\frac p2}\,|f_j|^p\Bigr)^{\frac2p}\\
&\ge
\Bigl(\sum_{j\in J}|t_j|^{\frac p2}\,\lambda^{\frac{(d-1)p}{2}}\,\lambda^{-k}\Bigr)^{\frac2p}\\
&\sim
\lambda^{\,d-1-\frac{2k}{p}}\,
\bigl\|t_j\bigr\|_{l^{\frac p2}(J)}.
\end{aligned}
\]
This establishes the sharpness of Theorem \ref{main theorem} for all
submanifolds of codimension at least two, and for hypersurfaces in the range
$4\leq p\leq\infty$, up to the stated logarithmic losses.

\begin{remark}
Finally, we note that realizing the sharpness of Theorem \ref{main theorem}
in the codimension-at-least-two regime requires a nontrivial upper bound on
$\#J$, or equivalently, a positive lower bound on $\beta$ (see, e.g., \eqref{lower beta}).
 By contrast, the $d=2$ case of Theorem \ref{main theorem} is sharp for \textbf{all} values of $\#J$.  

  Indeed, take a codimension-2 submanifold $\Sigma\subset \mathbb S^d$ and let $Q=E_\lambda$, so that
  \[
    \#J=\dim E_\lambda\sim\lambda^{d-1}.
  \]
  The sharp pointwise Weyl law on the sphere gives
  \[
    \sum_{\lambda_j\in I_\lambda}\bigl|e_{\lambda_j}(x)\bigr|^2
    \sim\lambda^{d-1},
    \quad x\in \mathbb S^d,
  \]
  and hence
  \[
    \Bigl\|\sum_{\lambda_j\in I_\lambda}|e_{\lambda_j}|^2\Bigr\|_{L^{p/2}(\Sigma)}
    \sim\lambda^{d-1}.
  \]
  On the other hand, Theorem \ref{main theorem} yields the upper bound
  \[
    \Bigl\|\sum_{\lambda_j\in I_\lambda}|e_{\lambda_j}|^2\Bigr\|_{L^{p/2}(\Sigma)}
    \lesssim
    \lambda^{{d-1}-\frac{2d-4}{p}+\frac{2(d-1)}{p}}
    =
    \lambda^{{d-1}+\frac{2}{p}}.
  \]
  It follows that, when $k\le d-2$
  the exponent in Theorem \ref{main theorem} strictly exceeds the true exponent $d-1$ unless $p=\infty$.
\end{remark}

To demonstrate the sharpness for the case $k = d-1$, $d \ge 3$, and $2\le p< 4$, we need the following equatorial harmonic lift example.
Let $d\geq3$, let $M=\mathbb S^d$, and let
$$
\Sigma=\mathbb S^{d-1}
=\{x\in\mathbb S^d:x_{d+1}=0\}.
$$
Let
$$
\lambda^2=l(l+d-1),\qquad l\in\mathbb N,
$$
and let $\{Y_j\}_{j=1}^{N_l}$ be an orthonormal basis of the space of spherical harmonics of degree $l$ on $\mathbb S^{d-1}$. Recall that
$$
N_l=\dim\mathcal H_l(\mathbb S^{d-1})
\sim l^{d-2}\sim\lambda^{d-2};
$$
see, for instance, \cite{DF2013}.

Write a point of $\mathbb S^d$ as
$$
x=(\sin\theta\,\omega,\cos\theta),
\qquad
\omega\in\mathbb S^{d-1},\quad 0\leq\theta\leq\pi,
$$
and define
$$
f_j(\theta,\omega)
=
c_l(\sin\theta)^lY_j(\omega),
\qquad 1\leq j\leq N_l,
$$
where $c_l>0$ is chosen so that
$$
\|f_j\|_{L^2(\mathbb S^d)}=1.
$$

Regarding $r^lY_j(\omega)$ as a homogeneous harmonic polynomial on $\mathbb R^d$ and extending it trivially in the last variable, we obtain a homogeneous harmonic polynomial of degree $l$ on $\mathbb R^{d+1}$. Hence $f_j$ is a spherical harmonic of degree $l$ on $\mathbb S^d$; see, e.g., \cite{DF2013}. It follows that $f_j\in E_\lambda$, and the family $\{f_j\}_{j=1}^{N_l}$ is orthonormal in $L^2(\mathbb S^d)$.

Moreover,
$$
c_l^{-2}
=
\int_0^\pi(\sin\theta)^{2l+d-1}\,d\theta
\sim l^{-1/2},
$$
and therefore
$$
c_l^2\sim l^{1/2}\sim\lambda^{1/2}.
$$
On the equator $\Sigma$, we have
$$
f_j|_\Sigma=c_lY_j.
$$
By the addition formula for spherical harmonics \cite{DF2013},
$$
\sum_{j=1}^{N_l}|Y_j(\omega)|^2
=
\frac{N_l}{|\mathbb S^{d-1}|}.
$$
Consequently,
$$
\sum_{j=1}^{N_l}|f_j(\omega)|^2
=
c_l^2\frac{N_l}{|\mathbb S^{d-1}|}
\sim
\lambda^{d-\frac32},
\qquad \omega\in\Sigma.
$$
Since the density on the left-hand side is constant on $\Sigma$, it follows that, for every $p\geq2$,
$$
\Biggl\|
\sum_{j=1}^{N_l}|f_j|^2
\Biggr\|_{L^{p/2}(\Sigma)}
\sim
\lambda^{d-\frac32}.
$$

We now compare this lower bound with Theorem \ref{main theorem}. If
$
2\leq p\leq\frac{2d}{d-1},
$
then
$$
2\delta(d-1,d,p)
=
\frac{d-1}{2}-\frac{d-2}{p},
\qquad
\frac1{\alpha(d-1,d,p)}
=
\frac12+\frac1p.
$$
Hence
$$
2\delta(d-1,d,p)
+
\frac{d-2}{\alpha(d-1,d,p)}
=
d-\frac32.
$$
Taking $t_j=1$ and using $N_l\sim\lambda^{d-2}$, we obtain the sharpness in this case.

For $\frac{2d}{d-1}\leq p<4$, we have
\[
2\delta(d-1,d,p)
=
d-1-\frac{2(d-1)}{p}
\]
and
\[
\frac{d-2}{\alpha(d-1,d,p)}
=
\frac{4d-p-4}{2p}.
\]
Hence
\[
2\delta(d-1,d,p)
+\frac{d-2}{\alpha(d-1,d,p)}
=
d-\frac32.
\]
Since $\#J\sim\lambda^{d-2}$, it follows that
\[
\lambda^{2\delta(d-1,d,p)}
(\#J)^{1/\alpha(d-1,d,p)}
\sim
\lambda^{d-\frac32},
\]
which agrees with the lower bound from the above construction.

\section{Sharpness on curves with non-vanishing curvature}\label{sec 8}
Define the $n$-th order highest-weight spherical harmonic
\[
H_n(x_1,x_2,x_3) = C_n(x_1 + ix_2)^n,
\]
where $C_n$ is the normalization constant such that $\|H_n\|_{L^2(\mathbb{S}^2)} = 1$. The eigenvalue of $H_n$ is $\lambda\sim n.$ This function is concentrated on the equator $\{(x_1,x_2,x_3) \in \mathbb{S}^2 : x_3 = 0\}$.
We also set
\[
    H_n^{\phi}(x)=C_n\bigl(x_1+i(x_2\cos\phi+x_3\sin\phi)\bigr)^n,
\]
the spherical harmonic $H_n$ rotated by an angle $\phi$ about the $x_1$-axis. Choose a bump function $\psi\in C_0^{\infty}(-1,1)$ such that $\int \psi=1$. Then define
\[
    u_{n}(x)=\int \psi(n^{1/3}\phi)\,H_n^{\phi}(x)\,d\phi,
\]
and set $\bar{u}_n(x)=u_n(x)/\|u_n(x)\|_{L^2(\mathbb{S}^2)}$. Note that $u_n$ is clearly a spherical harmonic. The construction of this eigenfunction comes from \cite[Remark 5.4]{bgt2007}, see also \cite{Tacy}.

Let $SO(3)$ be the special orthogonal group in three dimensions. For any $Q\in SO(3)$, define
\[
    H_{n,Q}(x)=H_n(Qx)\quad\text{and}\quad H^{\phi}_{n,Q}(x)=H^{\phi}_n(Qx).
\]
Then $H^{\phi}_{n}(x)=H_{n,\bar{Q}}(x)$, where $\bar{Q}\in SO(3)$ denotes the clockwise rotation by angle $\phi$ about the $x_1$-axis.

We require the following lemma, which appears in Lemma 8 of \cite{Han}. 

\begin{lemma}\label{inner-product}
Let $Q_1,Q_2\in SO(3)$. Suppose that $H_{n,Q_1}$ and $H_{n,Q_2}$ concentrate on the great circles $S_1$ and $S_2$  with respect to poles $p_1$ and $p_2$. 
Let $\alpha$ be the angle between the poles $p_1$ and $p_2$.
Then
\[
  \bigl|\langle H_{n,Q_1},H_{n,Q_2}\rangle_{\mathbb{S}^2}\bigr|
  = \cos^{2n}\Bigl(\frac{\alpha}{2}\Bigr).
\]
\end{lemma}

Let $\theta=(\theta_1,\theta_2)$ be a spherical coordinate system with $\theta_2\in[0,\pi]$ and $\theta_1\in[0,2\pi]$, and
\[
\begin{cases}
    x_3=\cos\theta_2,\\
    x_2=\sin\theta_2\sin\theta_1,\\
    x_1=\sin\theta_2\cos\theta_1.
\end{cases}
\]

\begin{lemma}\label{concentrate}
We have $|\bar{u}_n(\theta_1,\theta_2)|\sim n^{1/3}$ on an $n^{-1/3}\times n^{-2/3}$ region centered at $(0,\pi/2)$, and $\bar{u}_n(\theta_1,\theta_2)$ decays exponentially outside an $n^{-1/3}$-neighborhood of the equator $\{\theta:\theta_1\in[0,2\pi],\,\theta_2=\pi/2\}$.
\end{lemma}
\begin{proof}
Using Lemma \ref{inner-product} and $\int\psi=1$,
\[
\|u_n\|_{L^2(\mathbb{S}^2)}^2
=\iint \psi(n^{1/3}\phi_1)\psi(n^{1/3}\phi_2)\cos^{2n}\Bigl(\tfrac{\phi_1-\phi_2}{2}\Bigr)\,d\phi_1d\phi_2.
\]
Since $\cos t = e^{-t^2/2+O(t^4)}$, we replace $\cos^{2n}\Bigl(\tfrac{\phi_1-\phi_2}{2}\Bigr)$ by $\exp\,\bigl(-\tfrac n4(\phi_1-\phi_2)^2\bigr)$ up to negligible error, so the integral localizes to $|\phi_1-\phi_2|\lesssim n^{-1/2}$. After the change of variables $\varphi_i=n^{1/3}\phi_i$ we get $\|u_n\|_{L^2(\mathbb{S}^2)}^2\sim n^{-5/6}$, hence $\|u_n\|_{L^2(\mathbb{S}^2)}\sim n^{-5/12}$.

At $(\theta_1,\theta_2)=(0,\pi/2)$ one has
$H_n^\phi\sim n^{1/4}$, so
$u_n(0,\pi/2)\sim n^{-1/12}$ and therefore
$|\bar u_n(0,\pi/2)|\sim n^{1/3}$.
Multiplying by $e^{-in\theta_1}$ removes the oscillation.For $|\phi|\le n^{-1/3}$ and $|\theta_2-\pi/2|\le \epsilon n^{-1/2}$,
\[
\bigl|\partial_{\theta_1}(e^{-in\theta_1}H_n^\phi)\bigr|\lesssim n^{7/12},\qquad
\bigl|\partial_{\theta_2}(e^{-in\theta_1}H_n^\phi)\bigr|\lesssim n^{11/12}.
\]
Averaging in $\phi$ and normalizing give the bounds $n^{2/3}$ and $n$,
respectively, for $e^{-in\theta_1}\bar u_n$, so on a rectangle with
$\Delta\theta_1\sim n^{-1/3}$ and $\Delta\theta_2\sim n^{-2/3}$ around
$(0,\pi/2)$ the value of $|e^{-in\theta_1}\bar u_n|$ stays within a
constant factor, giving $|\bar u_n|\sim n^{1/3}$ there.

Finally, since
\[
u_n(x)=\int_{|\phi|\lesssim n^{-1/3}}\psi(n^{1/3}\phi)\,H_n^\phi(x)\,d\phi,
\]
it is an $n^{-1/3}$--wide superposition of highest-weight spherical
harmonics. For each $\phi$, $H_n^\phi$ decays exponentially once
$\operatorname{dist}(x,\Gamma_\phi)\gtrsim n^{-1/2}$, where $\Gamma_\phi$
is its great circle. Hence, if
$\operatorname{dist}(x,\Gamma_0)\ge A n^{-1/3}$ for sufficiently large
$A$, then for all contributing $\phi$ we have
$\operatorname{dist}(x,\Gamma_\phi)\gtrsim n^{-1/3}$, so the integrand is
exponentially small. Therefore $|u_n(x)|$, and thus $|\bar u_n(x)|$,
decays exponentially outside an $A n^{-1/3}$--neighborhood of the equator.

\end{proof}

For the sake of simplicity, we assume $\Sigma$ to be the latitude circle
\[
\gamma=\{(\theta_1,\theta_2):\ \theta_2=\pi/4\}.
\]
The general case of a curve with non-vanishing geodesic curvature on $\mathbb S^2$ is similar.

\begin{lemma}\label{distance}
On the unit sphere $\mathbb{S}^2$, let $p_1$ and $p_2$ lie on a latitude circle $\gamma$ of latitude $\pi/4$, and suppose their longitudes differ by $\alpha$. Let $S_1$ and $S_2$ be the great circles passing through $p_1$ and $p_2$, respectively, and tangent to $\gamma$ at those points. Then the angle $\theta$ between $s_1$ and $s_2$ at their intersection point equals the spherical distance $\mathrm{d}_{\mathbb{S}^2}(p_1,p_2)$.
\end{lemma}

\begin{proof}
This can be verified explicitly by an elementary coordinate computation on $\mathbb{S}^2$. More geometrically, it is easy to see that the angle is, to first order, proportional to the distance: work in normal geodesic coordinates centered at $p_1$ so that $S_1$ coincides with the $x_1$-axis. Since $\gamma$ has nonzero geodesic curvature, it is locally quadratic (parabolic) in $x_1$, whereas $S_2$ is, to first order at $p_2$, just its tangent line.

\end{proof}

Let $0<\varepsilon\ll1$. Let $C>0$ be sufficiently large constant. 
Choose a $C\lambda^{-1/6+\varepsilon}$–separated set of points $\{p_j\}_{j\in J}\subset\gamma$ so that for all $i\ne j$,
\[
C\,\lambda^{-1/6+\varepsilon}\le \mathrm{d}_{\gamma}(p_i,p_j)\le \tfrac{\pi}{4},
\]
where $\mathrm{d}_{\gamma}(p_i,p_j)$ is the arc length along $\gamma$. In this configuration,
\[
\#J=O\bigl(C^{-1}\lambda^{1/6-\varepsilon}\bigr).
\]
For each $p_i$ let $S_i$ be the great circle tangent to $\gamma$ at $p_i$. We construct $L^2$–normalized spherical harmonics concentrating near $p_i$ by
\[
u_i(x)=\bar{u}_{n}(Q_i x),
\]
where $Q_i\in SO(3)$ is chosen so that $Q_iS_i=\{x\in\mathbb{S}^2:\ x_3=0\}$ and $Q_ip_i=(1,0,0)$. Let $\theta_{ij}$ denote the angle between $S_i$ and $S_j$. Then we have the following lemma.

\begin{lemma}
When $C>0$ sufficiently large, we have for any $i\ne j$
\[
|\langle u_i,u_j\rangle_{\mathbb{S}^2}|\;\lesssim\; e^{-\lambda^{1/6}}.
\]
\end{lemma}

\begin{proof}
For $k=i,j$, let $S_k$ be the great circle tangent to $\gamma$ at $p_k$, and let $l_k$ be the common tangent line to $S_k$ and $\gamma$ at $p_k$. By Lemma \ref{distance},
\[
C\,\lambda^{-1/6+\varepsilon}\le \mathrm{d}_{\gamma}(p_i,p_j)\sim \mathrm{d}_{\mathbb{S}^2}(p_i,p_j)=\theta_{ij}.
\]
 Hence
\[
\mathrm{dist}(p_j,l_i)=\mathrm{dist}(p_i,l_j)\sim \mathrm{d}_{\gamma}(p_i,p_j)\sin\frac{\theta_{ij}}2
\sim \bigl(\mathrm{d}_{\gamma}(p_i,p_j)\bigr)^2
\ge C^2\lambda^{-1/3+2\varepsilon}\gg \lambda^{-1/3}.
\]

By Lemma \ref{inner-product} and Lemma \ref{concentrate} we finish the proof.
\end{proof}

By the above lemma and the argument from the previous section, there exists a matrix $B=[b_{ij}]$ such that
\[
g_i=\sum_{j\in J} b_{ij}\,u_j,
\qquad
\bigl|b_{ij}-\delta_{ij}\bigr|\lesssim \lambda^{1/6-\epsilon}e^{-\lambda^{1/6}}\lesssim e^{-c\lambda^{1/12}},
\]
and the family $\{g_j\}_{j\in J}$ is orthonormal. For any $x\in N_{\lambda^{-1/3}}(p_i)\cap\gamma$ we have
\[
\Bigl|\sum_{j\ne i} b_{ij}\,u_j(x)\Bigr|
\le C'\,(\#J)\,\lambda^{1/3} e^{-c\lambda^{1/12}}
\ll \lambda^{1/3},
\]
hence $|g_i(x)|\sim \lambda^{1/3}$ on that set. Since $\mathrm{d}_{\gamma}(p_i,p_j)\ge C\lambda^{-1/6+\varepsilon}\gg \lambda^{-1/3}$ for $i\ne j$, the neighborhoods $N_{\lambda^{-1/3}}(p_i)\cap\gamma$ are disjoint. Therefore, for $2\le p\le4$, taking $t_j\ge 0$, we have
\[
\begin{aligned}
\Bigl\|\sum_{j\in J} t_j\,|g_j|^{2}\Bigr\|_{L^{p/2}(\gamma)}
&\ge
\Biggl(\sum_{j\in J}\int_{N_{\lambda^{-1/3}}(p_j)\cap\gamma}
|t_j|^{p/2}\,|g_j|^{p}\,d\sigma\Biggr)^{2/p}\\
&\gtrsim
\Bigl(\sum_{j\in J}|t_j|^{p/2}\,\lambda^{p/3}\,\lambda^{-1/3}\Bigr)^{2/p}\\
&\sim
\lambda^{\frac{2}{3}-\frac{2}{3p}}\,
\|\{t_j\}\|_{l^{p/2}(J)}.
\end{aligned}
\]

\bibliography{mybibfile}
\bibliographystyle{alpha}

\end{document}